\documentclass[10pt, reqno]{amsart}
\usepackage{graphicx, amssymb, amsmath, amsthm}
\usepackage[foot]{amsaddr}
\usepackage{epsfig}
\usepackage{hyperref}
\usepackage[usenames, dvipsnames]{color}
\numberwithin{equation}{section}

\def\ov{\overline}

\newcommand{\R}{\mathbb{R}}

\newcommand{\Z}{\mathbb{Z}}

\newcommand{\T}{\mathbb{T}}

\newtheorem{theorem}{Theorem}[section]

\newtheorem{lemma}[theorem]{Lemma}

\newtheorem{proposition}[theorem]{Proposition}

\theoremstyle{definition}

\newtheorem{remark}[theorem]{Remark}

\makeatletter
\newcommand{\Extend}[5]{\ext@arrow0099{\arrowfill@#1#2#3}{#4}{#5}}
\makeatother

\title[Internal controllability of non-localized solution for the Kadomtsev-Petviashvili II equation ]{Internal controllability of non-localized solution for the Kadomtsev-Petviashvili II equation }
\author[I.~Rivas, C-M.~Sun]{Ivonne Rivas $^{1}$, Chenmin Sun$^{2}$ } 

\thanks{1. Some part of  this article was done during the visit to the Laboratoire J.A. Dieudonn\'e, supported by the European Research Council, ERC-2012-ADG, project number 320845: Semi classical Analysis and Partial Differential Equations}

\thanks{2. Supported by the European Research Council, ERC-2012-ADG, project number 320845: Semi classical Analysis and Partial Differential Equations}

\address{Universit\'e C\^{o}te d'Azur, LJAD,Nice, France$^1$ and Universidad del Valle, Cali, Colombia$^2$.}
\email{ivonne.rivas@correounivalle.edu.co$^1$}
\email{csun@unice.fr$^2$}

\begin{document}
 \maketitle

\begin{abstract}
The internal control problem for the Kadomstev-Petviashvili II equation, better known as KP-II,  is the object of studying in this paper. The controllability in $L^2(\T^2)$ from vertical strip is proved using the Hilbert Unique Method through the techniques of semiclassical and microlocal analysis.  Additionally,  a negative result for the controllability in  $L^2(\T^2)$ from horizontal strip is also showed.   \\
\end{abstract}

\section{Introduction} 
%%\textcolor{red}{Q} 这是对文字加颜色的命令
The Kadomtsev-Petviashvili equations better known as KP is  
\begin{equation}\label{kpall}  \partial_x(\partial_tu+\partial_x^3u+u\partial_x u)\pm\partial^2_{y}u=0 \end{equation}
and it was introduced  by  Kadomtsev and Petviashvili (see \cite{KP}) in 1970  from the study of  transverse stability of the solitary wave solution of the Korteweg-de Vries (KdV) equation.  The  KP equations are completely integrable and they can be solved by inverse scattering transform. Moreover, the equation \eqref{kpall} has been studied separately depending  on the sign that is used,  with a negative sign it is known as KP-I equation, otherwise it is the KP-II equation.  The propagation of the trajectories  behave very differently from one equation to another one and they do not allow us to study at the same time.   In this paper, we concentrate on the KP-II equation.\\   

Concerning to the Cauchy problem, the KP-II equation has been well studied.
In a pioneering work,  Bourgain \cite{B93} proved the global well-posedness of KP-II equation in $L^2(\mathbb{T}^2)$ by using the Fourier restriction norm introduced by himself in \cite{B-Gafa}. For non-periodic setting, Takaoka and Tzvetkov in \cite{Tzvetkov} proved local well-posedness in anisotropic Sobolev space $H^{s_1,s_2}(\mathbb{R}^2)$ with $s_1>-\frac{1}{3}$ and $s_2\geq 0$.  Hadac, Kerr and Koch in \cite{Koch} proved global well-posedness and scattering for small data in critical functional space $H^{-\frac{1}{2},0}(\mathbb{R}^2)$.  Molinet, Saut and Tzvetkov  in \cite{Saut} showed  the local and global well-posedness for partially periodic data .
 
 We will address the problem of exact controllability for KP-II equation. Before getting into this problem, we observe that \eqref{kpall} can be written as
 $$ \partial_tu+\partial_x^3u+u\partial_xu\pm\partial^{-1}_{x}\partial_y^2u=0,
 $$ 
 where the Fourier multiplier $\partial_x^{-1}$ is defined by 
 $$ \widehat{\partial_x^{-1}v}(k,\eta)=\frac{1}{ik}\widehat{v}(k,\eta)
 $$
 for all distributions with horizontal mean value 
 $$v\in\mathcal{D}'_0(\mathbb{T}^2):=\{v\in\mathcal{D}'(\mathbb{T}^2):\widehat{v}(0,l)=0\textrm{ for all }l\in\mathbb{Z}\}.
 $$
For any $s\in\mathbb{R},$ we  denote by $H_0^s(\mathbb{T}^2):=H^s(\mathbb{T}^2)\cap\mathcal{D}_0'(\mathbb{T}^2)$, a closed subspace of $H^s(\mathbb{T}^2)$. In particular,  $L_{0}^{2}(\mathbb{T}^2):= H^{0}_{0}(\mathbb{T}^2)$.  Additionally, for an open set $\omega$, define $$C_{\omega}^2:=\{g\in C^{2}(\mathbb{T}^2):\   \ g(x,y) \neq 0,  (x,y)\in \omega  \subset \mathbb{T}^2 , \text{ otherwise }  g(x,y) = 0, \text{outside } \omega  \}.$$

The internal control problem that we are interested  in this paper is:
given $T>0$ and  $u_0 ,u_1\in L_0^2$, does there exist a control input $f\in L^2((0,T);L^2(\mathbb{T}^{2})$, supported on some open subset $\omega\subset \T^2$, such that the solution of 
\begin{align} \label{exactcontrol}
\begin{cases}    \partial_tu+\partial_x^3u+\partial_x^{-1}\partial_y^2u + u\partial_{x}u=f,\quad
(t,x,y)\in\R\times\mathbb{T}^2,
\\
u|_{t=0}=u_0\in L_0^2(\mathbb{T}^2),
\end{cases}
\end{align}
satisfies $u(T,\cdot)=u_1$?\\

Additionally, we face the difficulty that the control input $f$ should be localized in $\omega$ while keeping the horizontal mean value. However, if the control region $\omega$ is either a horizontal strip or a vertical strip, we can define the control operator as follows.

For a vertical control region of the form $\omega=(a,b)\times\mathbb{T}$, we fix a non-negative real-valued function $g\in C_{\omega}^2(\mathbb{T})$  such that $(a,b)=\{x\in\mathbb{T}:g(x)>0\}$ and 
$\int_{\mathbb{T}}g=1$. In this case, we define the control input $\mathcal{G}h$, where $\mathcal{G}$ is the linear operator:
\begin{equation}\label{verticalcontrol}
\mathcal{G}h(x,y):=g(x)\left(h(x,y)-\int_{\mathbb{T}}g(x')h(x',y)dx'\right).
\end{equation} 
If the control region is a horizontal strip of the form $\omega=\mathbb{T}\times (a,b)$, we define the control input as $\mathcal{K}h$, where $\mathcal{K}$ is the operator:
\begin{equation}\label{horizontalcontrol}
\mathcal{K}h(x,y):=g(y)\left(h(x,y)-\int_{\mathbb{T}}g(y')h(x,y')dy'\right).
\end{equation}
%One can easily see that $\mathcal{G}h\in \mathcal{D}_0'$ if $h\in\mathcal{D}_0'$.

Our first result concerns the internal controllability of the linearized KP-II equation on vertical region:
\begin{align} \label{exactcontrollinear}
\begin{cases}    \partial_tu+\partial_x^3u+\partial_x^{-1}\partial_y^2u  =\mathcal{G}h,\quad
(t,x,y)\in\R\times\mathbb{T}^2,
\\
u|_{t=0}=u_0\in L_0^2(\mathbb{T}^2).
\end{cases}
\end{align}

\begin{theorem}\label{positive}
Given $T>0$, and $u_{0}, u_{1}\in L_0^{2}(\T^2)$, there exists  $h \in L^{2}((0,T);L^{2}(\T^2))$, such that the solution $u$ of \eqref{exactcontrollinear} satisfies $u(T)=u_{1}$.
\end{theorem}
For the nonlinear control system,
\begin{align} \label{exactcontrol-vertical}
\begin{cases}    \partial_tu+\partial_x^3u+\partial_x^{-1}\partial_y^2u + u\partial_{x}u=\mathcal{G}h,\quad
(t,x,y)\in\R\times\mathbb{T}^2,
\\
u|_{t=0}=u_0\in L_0^2(\mathbb{T}^2),
\end{cases}
\end{align}
by adapting a pertubative argument, relying on the Cauchy theory for the KP-II equation, we obtain the following result of the exact controllability in a local sense.
\begin{theorem}\label{nonlinear}
Given $T>0$, there exists $R > 0$ such that for any $u_0, u_1\in  L_0^2(\mathbb{T}^2)$
	satisfying  $\|u_0\|_{L^2(\mathbb{T}^2)}\leq R$ and $\|u_1\|_{L^2(\mathbb{T}^2)}\leq R$, there exists a control $h\in L^2((0,T);L^2(\mathbb{T}^2)) $, such that the solution $u$ of $\eqref{exactcontrol-vertical}$ with $\mathcal{G}$ satisfies $u(T) = u_1.$
\end{theorem} 

\begin{remark}   
	In \cite{B93}, Bourgain proved that the KP-II equation is globally well-posed in $H_0^s(\mathbb{T}^2)$ for all $s\geq 0$. Our results  in Theorem \ref{positive} and \ref{nonlinear} also hold for any data in $H_0^s(\mathbb{T}^2)$. The  main reason  of considering  $L^2(\mathbb{T}^2)$ here is that the quantity
	$$ \int_{\mathbb{T}^2}|u(t,x,y)|^2dxdy
	$$
	is conserved along the  KP-II flow \eqref{kpall} and hence $L^2(\mathbb{T}^2)$ is a natural functional space to study the problem of controllability.
\end{remark}

On the contrary, for the controllability from the horizontal region 
\begin{align} \label{exactcontrollinear-horizontal}
\begin{cases}    \partial_tu+\partial_x^3u+\partial_x^{-1}\partial_y^2u  =\mathcal{K}h,\quad
(t,x,y)\in\R\times\mathbb{T}^2,
\\
u|_{t=0}=u_0\in L_0^2(\mathbb{T}^2),
\end{cases}
\end{align}
we have a negative answer which shows that the exact controllability for linearized KP-II equation cannot hold at any time $T>0$ when the control region is a horizontal strip. 
\begin{theorem}\label{negative}
Given $T>0$ and $u_0\in L_0^2(\T^2)$, there exists $u_1 \in L^2(\mathbb{T}^2) $ and it does not exist $h\in L^2((0,T);L_0^2(\mathbb{T}))$, such that the solution $u$ of \eqref{exactcontrollinear-horizontal}  satisfies $u(T)=u_1$.
\end{theorem} 
The proofs of Theorem \ref{positive} and Theorem \ref{negative} rely on the propagation of singularities for the KP-II flow. It turns out that the propagation on the horizontal direction is much stronger than the vertical direction.  The heuristic is that the singularities will travel into some vertical control region in a very short time, however for a horizontal control region the singularities moves too slow to enter. This can be interpreted physically, since the KP equations describe the regime where the wavelengths in the transverse direction (in $y$) are much larger than
in the direction of propagation (in $x$). \\

%An interesting problem is to extend our result to the case where the control input is supported on a general non empty open set $\omega$. The difficulty will be a good set-up of the problem. Probably one should look for the control problem directly for the full KP-II equation
%$$ \partial_x(\partial_tu+\partial_x^3u)+\partial_y^2u=\mathcal{G}h
%$$ 
%instead of the non-local version \eqref{exactcontrollinear}. Once this is done, we believe that the following formal criteria for the exact controllability is valid, although further efforts are needed to prove it:
%\begin{enumerate}
%	\item If the control region $\omega$ satisfies that any horizontal geodesic will enter it before some time $T_0>0$, then \eqref{exactcontrollinear} is exactly controllable for any time $T>0$.
%	\item If there is a horizontal geodesic which does not intersect with $\omega$, then the exact controllability for \eqref{exactcontrollinear} cannot hold for any time $T>0$.
%\end{enumerate}

%There are other natural questions. The first one concerns about extending Theorem \ref{nonlinear} to large data. Secondly, the internal control problem for KP-I equation must also be understood. These problems will be considered a  forthcoming work.

The paper is organized as follows.
In Section 2,  some results of well-posedness are mentioned,  they will recover importance in the proof of the controllability of the full control system.  In Section 3, the linear controllability is established by proving the observability inequality. In Section 4, the local controllability of the nonlinear equation is proved by  fixed point arguments. In Section 5, we construct a counterexample to complete the proof of Theorem \ref{negative}.

\subsection*{Acknowledgement} The authors would like to thank Professor Gilles Lebeau, Ph.D. advisor of  author$^{1}$, for his productive discussions. The authors are grateful to Professor Lionel Rosier for his valuable comments on the previous version of this article.  

\section{Notations and Preliminaries}
Throughout this article, we use the identification $\mathbb{T}=\mathbb{R}/(2\pi\mathbb{Z})=[-\pi,\pi]\slash\mathbb{Z}_2$.   We will adapt the standard convention for constancy in PDE. The constant $C$ will denote a positive constant that can change from line to line and the denpendency will be specified if there is any risk of confusing.

We need the following classical inequality of Ingham.
\begin{proposition}[\cite{KL-book}]\label{newIngham}
	Let $(\omega_{k})_{k\in \Z}$ be a family of real numbers, satisfying the uniform gap condition
	$$ \gamma:=\inf_{k_1\neq k_2}|\omega_{k_1}-\omega_{k_2} |>0.
	$$ 
	If $I\subset \R$ is a bounded interval of length $|I|>\frac{2\pi}{\gamma}$, then there exists $C_{\gamma}>0$, depending only on $\gamma$ and the length $|I|$, such that for all $(a_k)_{k\in\Z}\subset l^2(\Z)$, we have 
	$$ \frac{1}{C_{\gamma}}\sum_{k\in\Z}|a_k|^2\leq \int_I\left|\sum_{k\in\Z}a_ke^{i\omega_k t} \right|^2dt\leq C_{\gamma}\sum_{k\in\Z}|a_k|^2.
	$$ 
\end{proposition}

Next we briefly review the Cauchy theory for KP-II following \cite{Saut}.
The initial value problem 
\begin{align} \label{kpII}
\begin{cases}    \partial_tu+\partial_x^3u+\partial_x^{-1}\partial^2_{y}u+ u \partial_{x}u=0,\quad
(t,x,y)\in\R\times\mathbb{T}^{2},
\\
u|_{t=0}=u_0\in L_0^2(\mathbb{T}^2),
\end{cases}
\end{align}
is proved in \cite{B93} by Bourgain to be  globally well-posed when $u_{0}\in H_0^{s}(\T^{2})$  for $s\ge 0$. \\

In \cite{B93}, Bourgain introduced a Fourier restriction norm
$$ \|u\|_{X^{s,b,b_1}}^2=\int_{\mathbb{R}}\sum_{(k,l)\in\mathbb{Z}^2}\left\langle\frac{\langle\sigma(\tau,k,l)\rangle}{\langle k\rangle^3}\right\rangle^{2b_1}\langle\sigma(\tau,k,l)\rangle^{2b}\langle(k,l)\rangle^{2s}|\widehat{u}(\tau,k,l)|^2d\tau,
$$
where $\sigma(\tau,k,l)=\tau-k^3+\frac{l^2}{k}$ and $\langle\cdot\rangle=\sqrt{1+|\cdot|^2}$. For $T>0$, the norm in the localized time interval $[0,T]$ is defined by
$$ \|u\|_{X_T^{s,b,b_1}}:=\inf\{\|w\|_{X^{s,b,b_1}}:w(t)=u(t) \, \mathrm{ on } \, (0,T)\}.
$$
Denote by $S(t)=e^{-it\left(\partial_x^3+\partial_x^{-1}\partial_y^2\right)}$ the linear semi-group, we have the following estimate:
\begin{proposition}\label{bilinear1}
	For $s\ge 0$, $-\frac{1}{2}<b'\leq 0<\frac{1}{2}<b\leq b'+1, \, b_{1}\in \R $ and $ T\leq 1$, we have
	$$  \left\|\int_0^tS(t-t')F(t')dt'\right\|_{X_T^{s,b,b_1}}\leq C T^{1-(b-b')}\|F\|_{X_T^{s,b',b_1}}.
	$$
	for any $F \in X_T^{s,b',b_1}$.
\end{proposition} 
The proposition above is false for the end points $b'=-\frac{1}{2}$ and $b=\frac{1}{2}$. However, for periodic problem, it seems that we cannot avoid to use these end points. The way to resolve this issue is to define an auxiliary norm
$$ \|u\|_{Z^{s,b}}:=\|\langle\sigma\rangle^{b-\frac{1}{2}}\langle(k,l)\rangle^{s}\widehat{u}\|_{l_{(k,l)}^2L_{\tau}^1}.
$$
We denote by $Z_T^{b,s}$ the restricted spaces, defined in the same manner. The analogue of Proposition \ref{bilinear1} is as follows:
\begin{proposition}\label{linearestimate} {Under the same conditions as in Proposition \ref{bilinear1}}, we have
	$$\left\|S(t)u_0+\int_0^tS(t-t')F(t')dt'\right\|_{X_T^{s,\frac{1}{2},b_1}\cap Z_T^{s,\frac{1}{2}}}\leq C \|u_0\|_{H^s}+C\|F\|_{X_T^{s,-\frac{1}{2},b_1}\cap Z_T^{s,-\frac{1}{2}}}.$$
\end{proposition}
The proof can be found, for example in \cite{Tao}. In order to show that the equation \eqref{kpII} is locally well-posed in the Fourier restriction spaces, we write it in the integral form: 
\begin{equation}\label{int_sol}
u(t)= S(t) u_{0}+ \int_{0}^{t} S(t-t') (u\partial_{x}u)(t') dt'.
\end{equation}
To use the fixed point argument, the following bilinear estimate is crucial:
\begin{proposition}[see \cite{Saut}]\label{bilinear}
	There exist $\frac{1}{4}<b_1<\frac{3}{8}$, $C>0$ and $\nu>0$ such that for all $0<T\leq 1, s\geq 0$, the following bilinear estimate holds
	$$ \|\partial_x(uv)\|_{X_T^{s,-\frac{1}{2},b_1}\cap Z_T^{s,-\frac{1}{2}}}\leq C T^{\nu} \|u\|_{X_T^{s,\frac{1}{2},b_1}}\|v\|_{X_T^{s,\frac{1}{2},b_1}}
	$$
	for functions $u,v\in X_T^{s,\frac{1}{2},b_1}$ satisfying
	$$ \int_{\mathbb{T}}u(t,x,y)dx=\int_{\mathbb{T}}v(t,x,y)dx=0.
	$$
	
\end{proposition}
This bilinear estimate is established by Bourgain in \cite{B93}. We use the adapted version of \cite{Saut}, in which the authors dealt with partially periodic data.

\section{Linear controllability on vertical strip}

In this section, the study of the internal controllability of linear system  \eqref{exactcontrollinear} is addressed by defining a linear operator in Proposition \ref{controloperator}, which characterizes the  control input of the linear system and drives the solution from an initial state $u_{0}$ to a final state $u_{1}$. Notice that by reversibility, the exact controllability is equivalent to null controllability:   given any initial state $u_0\in L_0^2$, find a function $h\in L^2((0,T)\times\mathbb{T}^2)$ so that the equation
satisfies $u(0,\cdot)=u_0$ and $u(T,\cdot)=0$. Hence, we will study the null controllability.\\

The classical strategy to study the null controllability is to show the observability inequality for the adjoint system associated to the equation, in the KP-II case, it  matches with the homogeneous linearized KP-II equation:
\begin{align} \label{linearcontrol}
\begin{cases}    \partial_tu+\partial_x^3u+\partial_x^{-1}\partial_y^2u =0,\quad
(t,x,y)\in\R\times\mathbb{T}^2,
\\
u|_{t=0}=u_0\in L_0^2(\mathbb{T}^2), 
\end{cases}
\end{align}

From the classical Hilbert Uniqueness Method (HUM), one can deduce that the null controllability is equivalent to the observability for its adjoint system.
\begin{proposition}[See \cite{Li}]\label{HUM}
	Given $T>0$, the system \eqref{exactcontrollinear}  is null controllable at $T$ if and only if given $u_0\in L^2(\mathbb{T}^2)$, there exists a unique solution $u$ to \eqref{linearcontrol}  such that 
	\begin{equation}\label{observability2D}
	\|u_0\|_{L^2(\mathbb{T}^2)}^2\leq C_T\int_0^T\int_{\mathbb{T}^2}|\mathcal{G}u(t,x,y)|^2dxdydt,
	\end{equation}
where the constant $C_T>0$ does not depend on $u_0$. 
\end{proposition}
The region where the control will be placed is a vertical strip given by  $$\omega:=]a,b[\times \T$$
and the operator $\mathcal{G}$ is given by \eqref{verticalcontrol}.
The region $\omega$ will allow us to get a reduction of the KP-II equation \eqref{linearcontrol} in one dimension. Indeed,
	by expanding the solution $u(t,x,y)$ to \eqref{linearcontrol} in Fourier series in $y$ variable
	$$ u(t,x,y)=\sum_{l\in\mathbb{Z}}a_{l}(t,x)e^{ily},
	$$
	we find that for each $l\in\mathbb{Z}$, $a_l$ satisfies the equation
	$$ \partial_ta_l+\partial_x^3a_l-l^2\partial_x^{-1}a_l=0.
	$$
	Therefore, by changing the notation, it is reduced to the study of following $\lambda$-dependent equations
	\begin{align} \label{1Dequation}
	\begin{cases}    \partial_tu+\partial_x^3u-\lambda^2\partial_x^{-1}u=0,\quad
	(t,x)\in\R\times\mathbb{T},
	\\
	u|_{t=0}=u_0\in L_0^2(\mathbb{T}).
	\end{cases}
	\end{align}

\subsection{Observability inequality}
Due to  Proposition \ref{HUM}, the proof of Theorem \ref{positive} is reduced to the proof of \eqref{observability2D}. From the one dimensional reduction and Plancherel's Theorem, we can further reduce the observability \eqref{observability2D} to the following uniform observability for the family of equations \eqref{1Dequation}.
\begin{proposition}\label{1Dobservability}
Given $T>0$,  there exists $C_T>0$ such that for all $\lambda>0$, 
	\begin{equation}\label{observable1D}
	\|u_0\|_{L^2(\mathbb{T})}^2\leq C_T\int_0^T\int_{\mathbb{T}}|\mathcal{G}u(t,x)|^2dxdt
	\end{equation}
	holds for all solution $u$ of \eqref{1Dequation}.
\end{proposition}

The rest of this section is devoted to the proof of Proposition \ref{1Dobservability}. The strategy is as follows. First,  we reduce the inequality \eqref{observable1D} to a weaker one \eqref{observabilityuniform1}, which  is the observability for high frequencies and  it does not consider the normalization part  which simplify the operator $\mathcal{G}$. Next, inspired by the work of Lebeau in \cite{Lebeau}, we rescale the time to change it to the semi-classical scale. This reduces the weak observability for system \eqref{1Dequation} to an inequality of the same form but for another semi-classical system \eqref{semiclassical1}. The third step is to reduce the inequality in the previous step to a frequency-localized one. Finally, we use the propagation argument to prove the frequency-localized semi-classical observability, namely \eqref{singlefrequency}. 

\subsubsection{\bf Reduction to weak observability}

The weak observability takes the form, uniformly in $\lambda\geq 0$,
\begin{equation}\label{observabilityuniform1}
\|u_0\|_{L^2(\mathbb{T})}^2\leq C_T\int_0^T\int_{\mathbb{T}}|g(x)u(t,x)|^2dxdt+C\|u_0\|_{H^{-1}(\mathbb{T})}^2.
\end{equation}

First, we prove a lemma concerning about the commutator of a high-frequency cut-off and the operator $\mathcal{G}$.
\begin{lemma}\label{commmutator}
	Take $\chi\in C^{\infty}(\mathbb{R})$ with $\mathrm{supp }(\chi)\subset \{|\xi|> 1\}$ and $\chi|_{|\xi|\geq 2}=1$. Then there exist $h_0>0, C>0$ such that for all $0<h<h_0$, we have
	$$ \int_0^T\|[\chi(hD_x),\mathcal{G}]u(t,\cdot)\|_{L^2(\mathbb{T})}^2dt\leq Ch^2\|u(0)\|_{L^2(\mathbb{T})}^2.
	$$ 
\end{lemma}
\begin{proof}
	We write
	$$ \int_0^T\|[\chi(hD_x),\mathcal{G}]u(t,\cdot)\|_{L^2(\mathbb{T})}^2dt\leq C(\mathrm{I}+\mathrm{II}),
	$$
	where
	\begin{equation*}
	\begin{split}
	&\mathrm{I}=\int_0^T\int_{\mathbb{T}}|[g(x),\chi(hD_x)]u(t,x,y)|^2dxdt,\\
	&\mathrm{II}=\int_0^T\int_{\mathbb{T}}\left|g(x)\int_{\mathbb{T}}g(x')\chi(hD_x)u(t,x')dx'-\chi(hD_x)\left(g(x)\int_{\mathbb{T}}g(x')u(t,x')dx'\right)\right|^2dxdt.
	\end{split}
	\end{equation*}
From symbolic calculus\footnote{Though $g$ is not assumed to be smooth, the following estimate  is still valid.}, we have
	$$ \|[g(x),\chi(hD_x)]\|_{L^2\rightarrow L^2}\leq Ch,
	$$
	and by conservation of $L^2$ norm, we have
	$$ \mathrm{I}\leq Ch^2\int_0^T\|u(t)\|_{L^2(\mathbb{T})}^2dt=Ch^2T\|u(0)\|_{L^2(\mathbb{T})}^2.
	$$
	
	For II, we first calculate (to simplify the notation, we omit the variable $t$ here)
	\begin{equation*}
	\begin{split}
	&\left(g(x)\int_{\mathbb{T}}g(x')(\chi(hD_x)u)(x')dx'\right)^{\widehat{}}(l)-\chi(hD_x)\left(g(x)\int_{\mathbb{T}}g(x')u(x')dx'\right)^{\widehat{}}(l)\\
	=&\widehat{g}(l)\sum_{l_1\neq 0}\left(\chi(hl_1)-\chi(hl)\right)\widehat{g}(l_1)\widehat{u}(l).
	\end{split}
	\end{equation*}
	Since
	$ |\chi(hl_1)-\chi(hl)|\leq \|\chi'\|_{L^{\infty}}h|l_1-l|,
	$
	we have
	\begin{equation*}
	\begin{split}
	\mathrm{II}\leq &Ch^2\sum_{l}|\widehat{g}(l)|^2\left|\sum_{l_1\neq 0}|l_1-l|\widehat{g}(l_1)\widehat{u}(l_1)\right|^2\\
	\leq &Ch^2\sum_{l}|\widehat{g}(l)|^2\left(\sum_{l_1\neq 0}|l_1-l|^2|\widehat{g}(l_1)|^2\right)\left(\sum_{l_1\neq 0}|\widehat{u}(l_1)|^2\right)\\
	\leq &Ch^2\|u\|_{L^2(\mathbb{T})}^2\sum_{l,l_1\neq 0}|l_1-l|^2|\widehat{g}(l_1)|^2|\widehat{g}(l)|^2\\
	=&Ch^2\|u\|_{L^2(\mathbb{T})}^2,
	\end{split}
	\end{equation*}
	where we used the fact that $g\in C^2(\mathbb{T})$.
\end{proof}

\begin{proposition}\label{firstreduction}
	\eqref{observabilityuniform1} implies the following full observability inequality
	\begin{equation}\label{observabilityuniform2}
	\|u_0\|_{L^2(\mathbb{T})}^2\leq C_T\int_0^T\int_{\mathbb{T}}|\mathcal{G}u(t,x)|^2dxdt.
	\end{equation}
\end{proposition}
\begin{proof}
	
	The proof is essentially a unique continuation  argument. However, it is more delicate since we need a uniform estimate with respect to  $\lambda$. The  proof will be divided into two steps.\\
	
	The first step is to show that for any fixed $\lambda>0$, \eqref{observabilityuniform2} holds with constant $C(\lambda)>0$ which may depend on $\lambda$. We argue by contradiction, assuming that \eqref{observabilityuniform2} is not true, then we can select a sequence $u_n$ of solutions to \eqref{1Dequation} so that 
	$$ \|u_n(0)\|_{L^2(\mathbb{T})}=1\quad \text{and} \quad \lim_{n\rightarrow\infty}\int_0^T\int_{\mathbb{T}}|\mathcal{G}u_n(t,x)|^2dxdt=0.
	$$ 
	Up to a subsequence, we may assume that $u_n(0)\rightharpoonup u_0$, weakly in $L^2(\mathbb{T})$. One can easily verify that $u_0\in L_0^2(\T)$. Moreover, from semi-group property, $u_n(t)\rightharpoonup
	u(t)$ weakly in $C([0,T];L^2(\mathbb{T}))$ and $u(t)$ is the distributional solution to \eqref{1Dequation} with initial data $u_0$. Since $\mathcal{G}: L_0^2(\T)\rightarrow L_0^2(\T)$ is a bounded operator, we have that $\mathcal{G}u(t,\cdot)=0 $ in $L_0^2(\mathbb{T})$ for a.e. $t\in[0,T]$. This means that $u(t,x)|_{\omega}=C(t)$ in $\mathcal{D}'(\omega)$ for a.e. $t\in[0,T]$. Moreover, from the strong continuity of the semi-group on $L_0^2(\T)$, 
	$$ C(t)=\int_{\mathbb{T}}g(x)u(t,x)dx,\forall t\in[0,T],
	$$ 
	and $C(t)$ is a continuous function in $t$. Therefore we have that
	$$ g(x)\left(u(t,x)-C(t)\right)=0, \textrm{ in }C([0,T];L^2(\mathbb{T})).
	$$
	Thus $u(t,x)|_{x\in\omega}=C(t)$ in $\mathcal{D}'(\omega)$ for all $t\in[0,T]$. Now, if we rewrite the equation \eqref{1Dequation} as $\partial_x(\partial_tu+\partial_x^3u)+\lambda u=0$ and evaluate $u$ for $x\in\omega$, we have that $u|_{\omega}=0$ in $\mathcal{D}'(\omega)$. Next we claim that $u\equiv 0$. Indeed,  following \cite{BLR}, we consider the following family of sets (depending on $T'$):
			$$ \mathcal{N}_{T'}:=\{u_0\in L_0^2(\mathbb{T}):S(t)u_0|_{\omega}=0,\forall t\in[0,T']\}.
			$$
			For any $T'>T/2$, applying inequality \eqref{observabilityuniform1} (with $T/2$), we have that for any $u_0\in \mathcal{N}_{T'}$,
			$$ \|u_0\|_{L^2(\T)}\leq C\|u_0\|_{H^{-1}(\mathbb{T})}.
			$$
			This implies that the subspace $\mathcal{N}_{T'}$ in $L_0^2(\mathbb{T})$ is finite dimensional. Moreover, $\mathcal{N}_{T_1}\subset \mathcal{N}_{T_2}$ if $T_1>T_2$. Now for $\delta>0$ small, $S(\delta): \mathcal{N}_{T'}\subset \mathcal{N}_{T'-\delta}$ is a linear mapping. Since for $T'>T/2$, dim$\mathcal{N}_{T'-\delta}<\infty$, there exists $\delta_0>0$ such that for all $\delta\leq \delta_0$, $\mathcal{N}_{T'-\delta}=\mathcal{N}_{T'}.$ Therefore, $(S(\delta)-\mathrm{Id})\delta^{-1}:\mathcal{N}_{T'}\rightarrow \mathcal{N}_{T'}$ is a linear mapping. Passing $\delta\rightarrow 0$, we have that $(\partial_tS(t))|_{t=0}:\mathcal{N}_{T'}\rightarrow \mathcal{N}_{T'}$. Denote by $\sigma$ and $v_0$, any of its eigenvalue and the corresponding eigenfunction of $\partial_tS(t)|_{t=0}$ on $\mathcal{N}_{T'}$, since $(\partial_{t}S(t)v_0)|_{t=0}=(-\partial_x^3+\lambda^2\partial_x^{-1})v_0$, we have 
			$$ (-\partial_x^3+\lambda^2\partial_x^{-1})v_0=\sigma v_0.
			$$
			This implies that $v_0$ has only finite number of non-vanishing Fourier modes. Thus $v_0$ has an analytic extension near the real axis. Therefore, $v_0|_{\omega}=0$ yileds $v_0\equiv 0$.
			Hence $\mathcal{N}_{T'}=\{0\}. $ 

%	Next we claim that $u\equiv 0$. Indeed, consider the following set:
%	$$ \mathcal{N}:=\{u_0\in L_0^2(\mathbb{T}):u(t,\cdot)|_{\omega}=0,\forall t\in[0,T]\}.
%	$$
%	Applying inequality \eqref{observabilityuniform1}, we have that
%	$$ \|u_0\|_{L^2(\T)}\leq C\|u_0\|_{H^{-1}(\mathbb{T})}
%	$$
%	for all $u_0\in\mathcal{N}$. This implies that the subspace $\mathcal{N}$ in $L_0^2(\mathbb{T})$ is finite dimensional.
	
%	 Thus, for any $u_0\in \mathcal{N}$, the solution can be written of the form
%	$$ u(t,x)=\sum_{1\leq |l|\leq M}a_le^{it(l^3-\lambda^2l^{-1})}e^{ilx}.
%	$$ 
%	This trigonometric polynomial is smooth and it vanishes in $\omega$. From classical result (see for instance \cite{Lebeaucour}), $a_l\equiv 0$ for all $1\leq |l|\leq M$. This implies that $u\equiv 0.$ 
	
	Since the weak limit of $u_n(0)$ is $0$, we have $$\int_{\mathbb{T}}g(x)u_n(t,x)dx\rightarrow 0 \text{ and } \|gu_n\|_{L^2([0,T]\times\mathbb{T})}\rightarrow 0.$$ Moreover, up to a subsequence, $\|u_n(0)\|_{H^{-1}(\mathbb{T})}\rightarrow 0$, due  to Rellich Theorem. This is a contradiction to the assumption that $\|u_n(0)\|_{L^2(\mathbb{T})}=1$.

	The second step is to prove that \eqref{observabilityuniform2} is uniformly on $\lambda$. Again, we assume that \eqref{observabilityuniform2} is not true. Then there exists a sequence of positive numbers $\lambda_n>0$ and a sequence of solutions $u_n$ to \eqref{1Dequation} with parameters $\lambda_n$ such that
	$$ \|u_n(0)\|_{L^2(\mathbb{T})}=1 \quad \text{and} \quad   \lim_{n\rightarrow \infty}\int_0^T\int_{\mathbb{T}}|\mathcal{G}u_n(t,x)|^2dxdt=0.
	$$
	Up to a subsequence, we may assume that $\lambda_n\rightarrow \lambda_{\infty}\in [0,\infty]$. Suppose $\lambda_{\infty}<\infty$, similar argument as in the first step  will lead to a contradiction. 

	The last possibility is   $\lambda_{\infty}=\infty$. We write
	$$ u_n(0)=\sum_{l\neq 0} a_{n,l}e^{ilx}
	$$ 
	and the corresponding solution of \eqref{1Dequation} is
	$$ u_n(t,x)=\sum_{l\neq 0}a_{n,l}e^{it\left(l^3-\frac{\lambda_n^2}{l}\right)}e^{ilx}.
	$$
	
	For any $\epsilon_0>0$, we set
	$$ u_{n}^{(\epsilon_0)}:=\sum_{|l|\geq \frac{1}{\epsilon_0}}a_{n,l}e^{it\left(l^3-\frac{\lambda_n^2}{l}\right)}e^{ilx},\quad v_n^{(\epsilon_0)}=u_n-u_n^{(\epsilon_0)}.
	$$
	From Lemma \ref{commmutator}, we have
	\begin{equation*}
	\begin{split}
	\int_0^T\|\mathcal{G}u_n^{(\epsilon_0)}(t)\|_{L^2(\mathbb{T})}^2dt\leq & C\epsilon_0^2\|u_n(0)\|_{L^2(\mathbb{T})}^2+C\int_0^T\|(\mathcal{G}u_n)^{(\epsilon_0)}(t)\|_{L^2(\mathbb{T})}^2dt.
	\end{split}
	\end{equation*} 
	Thus, there exists $C>0$ such that for any $\epsilon_0>0,$ we have 
	\begin{equation}\label{1}
	\begin{split}
	&\limsup_{n\rightarrow\infty}\int_0^T\|\mathcal{G}u_n^{(\epsilon_0)}(t)\|_{L^2(\mathbb{T})}^2dt\leq C\epsilon_0^2,\\
	&\limsup_{n\rightarrow\infty}\int_0^T\|\mathcal{G}v_n^{(\epsilon_0)}(t)\|_{L^2(\mathbb{T})}^2dt\leq C\epsilon_0^2.
	\end{split}
	\end{equation}
	For any $\epsilon>0$ small, we can find $\epsilon_0>0$  small enough such that 
	$$ \sum_{|l|\geq \frac{1}{\epsilon_0}}|\widehat{g}(l)|^2\leq\epsilon^2,
	$$
	and then
	$$ \left \|g(x)\int_{\mathbb{T}}g(x')u_n^{(\epsilon_0)}(t,x')dx'\right\|_{L^2(\mathbb{T})}^2\leq \epsilon^2\|g\|_{L^2(\mathbb{T})}^2\|u_n^{(\epsilon_0)}(0)\|_{L^2(\mathbb{T})}^2.
	$$
	Thus, from \eqref{observabilityuniform1},\begin{equation*}
	\begin{split}
	\|u_n^{(\epsilon_0)}(0)\|_{L^2(\mathbb{T})}^2\leq &\, C\epsilon^2+C\epsilon_0^2+\|u_n^{(\epsilon_0)}(0)\|_{H^{-1}(\mathbb{T})}^2\\
	\leq &\, C(\epsilon^2+\epsilon_0^2),
	\end{split}
	\end{equation*}
for $n$ large enough.\\
	
	On the other hand, direct calculation yields
	\begin{equation*}
	\begin{split}
	\int_0^T\|\mathcal{G}v_n^{(\epsilon_0)}(t)\|_{L^2(\mathbb{T})}^2dt=&\int_0^T\sum_l\left|\sum_{1\leq|l_1|\leq 1/\epsilon_0}(\widehat{g}(l-l_1)-\widehat{g}(l)\widehat{g}(l_1))a_{n,l_1}e^{it\left(l_1^3-\frac{\lambda_n^2}{l_1}\right)}\right|^2dt\\
	\geq & C\sum_l\sum_{1\leq|l_1|\leq 1/\epsilon_0}|\widehat{g}(l-l_1)-\widehat{g}(l)\widehat{g}(l_1)|^2|a_{n,l_1}|^2\\
	=&C\sum_{1\leq |l_1|\leq 1/\epsilon_0}c_{l_1}|a_{n,l_1}|^2
	\end{split}
	\end{equation*}
	with $c_{l_1}=\sum_l|\widehat{g}(l-l_1)-\widehat{g}(l)\widehat{g}(l_1)|^2$, by the Ingham inequality (Proposition \ref{newIngham}), due to the assumption that $\lambda_n\rightarrow\infty$. Notice that the constant $C$ can be chosen independent of $n$ and $\epsilon_0$, provided that if $n$ large enough then for any $l_1\neq l_2, 1\leq |l_1|,|l_2|\leq \frac{1}{\epsilon_0}$,
	\begin{equation*}
\Big|\big(l_1^3-\frac{\lambda_n^2}{l_1}\big)-\big(l_2^3-\frac{\lambda_n^2}{l_2}\big) \Big|=\Big|(l_1-l_2)\big(l_1^2+l_1l_2+l_2^2+\frac{\lambda_n^2}{l_1l_2}\big)\Big|\geq \gamma>0 \quad \text{and} \quad  T>\frac{2\pi}{\gamma}.
	\end{equation*}
	Note that $c_{l_1}\geq |\widehat{g}(0)-\widehat{g}(l_1)^2|^2$  and  $\widehat{g}(0)=1$, there exists a constant $c_0>0$, independent of $\epsilon_0, \epsilon$ and $n$, so that $c_{l_1}\geq c_0$ for all $1\leq |l_1|\leq 1/\epsilon_0$. Thus, for $n$ sufficiently large,
	$$ \|v_n^{(\epsilon_0)}(0)\|_{L^2(\mathbb{T})}^2\leq \frac{C}{c_0}\int_0^T\|\mathcal{G}v_n^{(\epsilon_0)}(t)\|_{L^2(\mathbb{T})}^2dt\leq C\epsilon_0^2.
	$$ 
	Therefore, 
	$$ 1=\limsup_{n\rightarrow\infty}\|u_n(0)\|_{L^2(\mathbb{T})}^2=\|u_n^{\epsilon_0}(0)\|_{L^2(\mathbb{T})}^2+\|v_n^{\epsilon_0}(0)\|_{L^2(\mathbb{T})}^2\leq C(\epsilon_0^2+\epsilon^2)<1,
	$$
	which cannot happen.
\end{proof}

\subsection{Reduction to semi-classical observability}

Now, we consider the semi-classical equation of the following form:
\begin{align} \label{semiclassical1}
\begin{cases}    h\partial_tu+(h\partial_x)^3u-(h\partial_x)^{-1}u=0,\quad
(t,x)\in\R\times\mathbb{T},
\\
u|_{t=0}=u_0\in L_0^2(\mathbb{T}).
\end{cases}
\end{align}

\begin{proposition}\label{semiclassicalreduction}
	Assume that there exist $T_0>0,h_0>0$ such that the following semi-classical observability
	\begin{equation}\label{semiclassicalobservability}
	\|u_0\|_{L^2(\mathbb{T})}^2\leq C_{T_0}\int_0^{T_0}\int_{\mathbb{T}}|g(x)u(t,x)|^2dxdt+C\|u_0\|_{H^{-1}(\mathbb{T})}^2
	\end{equation}
	holds for any solution $u$ of \eqref{semiclassical1} with initial data $u_0\in L_0^2(\T)$, uniformly for $0<h<h_0$. Then for any $T>0$, the observability inequality \eqref{observabilityuniform1} holds.
\end{proposition}
\begin{proof}
	It would be sufficient to prove \eqref{observabilityuniform1} when $\lambda>1$ is large enough since for bounded $\lambda\geq 0$, the equation \eqref{1Dequation} can be viewed as a pertubation of linear KdV equation and the constant $C$ on the right hand side of \eqref{observabilityuniform1} can be chosen to be continuously depended on $\lambda$. For $\lambda\geq \frac{1}{h_0^2}$, we write $\lambda^2=\frac{1}{h^4}$ and \eqref{1Dequation} becomes
	$$ h^3\partial_tu+(h\partial_x)^3u-(h\partial_x)^{-1}u=0.
	$$
	 Setting $w(t,x)=u(h^{2}t,x)$,  it satisfies the equation
	$$ h\partial_tw+(h\partial_x)^3w-(h\partial_x)^{-1}w=0.
	$$
	Now from \eqref{semiclassicalobservability}, we have
	$$ \|w(0)\|_{L^2(\mathbb{T})}^2\leq C\int_0^{T_0}\int_{\mathbb{T}}|g(x)w(t,x)|^2dxdt+C\|w(0)\|_{H^{-1}(\mathbb{T})}^2.
	$$
	Changing back to $u(t,x)$, it holds
	$$ \|u(0)\|_{L^2(\mathbb{T})}^2\leq \frac{C}{h^2}\int_0^{h^2T_0}\int_{\mathbb{T}}|g(x)u(s,x)|^2dxds+C\|u(0)\|_{H^{-1}(\mathbb{T})}^2.
	$$
	Due to the invariance of the time-translation  and the conservation of $H^s$-norm of the linear equation, we have for any $M\in\mathbb{N}$,
	\begin{equation*}
	\begin{split}
	\|u(Mh^2T_0)\|_{L^2(\mathbb{T})}^2&=\|u(0)\|_{L^2(\mathbb{T})}^2\\ \leq &\, \frac{C}{h^2}\int_{Mh^2T_0}^{(M+1)h^2T_0}\int_{\mathbb{T}}|g(x)u(s,x)|^2dxds+C\|u(Mh^2T_0)\|_{H^{-1}(\mathbb{T})}^2\\
	=&\,\frac{C}{h^2}\int_{Mh^2T_0}^{(M+1)h^2T_0}\int_{\mathbb{T}}|g(x)u(s,x)|^2dxds+C\|u(0)\|_{H^{-1}(\mathbb{T})}^2.
	\end{split}
	\end{equation*} 
	Summing for $M$ from $0$ to $\epsilon_0h^{-2}$, with $\epsilon_0 T_0\leq T$, we have
	$$ \|u(0)\|_{L^2(\mathbb{T})}^2\leq \frac{C}{\epsilon_0}\int_0^T\int_{\mathbb{T}}|g(x)u(t,x)|^2dxdt+\frac{C}{\epsilon_0}\|u(0)\|_{H^{-1}(\mathbb{T})}^2.
	$$
	This completes the proof of Proposition \ref{semiclassicalreduction}.
\end{proof}

\subsubsection{\bf Reduction to frequency localized semi-classical observability}

We use an homogeneous Littlewood-Paley decomposition. Take $\psi\in C_c^{\infty}(\mathbb{R})$ with support supp$(\psi)\subset \{1/2\leq |\xi|\leq 2\}$ and $\psi_k\in C_c^{\infty}(\mathbb{R})$ such that
$$ \sum_{k\in\mathbb{Z}}\psi_k(\xi)=1,\forall\xi\neq 0,
$$
where $\psi_k(\xi)=\psi(2^{k}\xi)$. We will reduce the proof of the inequality \eqref{semiclassicalobservability} to the following:
\begin{proposition}\label{singlefrequencyproposition}
	There exist $\epsilon_0>0,h_0>0,$ small and $T_0>0,$ $C_0=C_0(\epsilon_0)>0$ such that for all $k\in\mathbb{Z}$, with $2^kh\leq\epsilon_0$,
	\begin{equation}\label{singlefrequency}
	\|\psi_k(hD_x)u(0)\|_{L^2(\mathbb{T})}^2\leq C_0\int_0^{T_0}\int_{\mathbb{T}}|g(x)\psi_k(hD_x)u(t,x)|^2dxdt
	\end{equation}
	holds for all solutions $u(t,x)$ of \eqref{semiclassical1}, uniformly in $h\in(0,h_0)$.
\end{proposition}

This proposition will be proved in the next subsection. In fact, from the proof, we can deduce that if Proposition \ref{singlefrequencyproposition} holds true for some $\epsilon_0>0,h_0>0$, it is also true for any other parameters $\epsilon_1,h_1$ such that $\epsilon_1<\epsilon_0$ and $h_1<h_0$ with possible change in the dependency of constant $C_0$.

\begin{lemma}
	Proposition \ref{singlefrequencyproposition} implies the inequality  \eqref{semiclassicalobservability}.
\end{lemma}
Indeed, applying Lemma \ref{commmutator}, we have
\begin{equation*}
\begin{split}
\|g\psi_k(hD_x)u\|_{L^2(\mathbb{T})}^2\leq & 2\|\psi_k(hD_x)(gu)\|_{L^2(\mathbb{T})}^2+ 2\|[\psi(2^khD_x),g]u\|_{L^2(\mathbb{T})}^2\\
\leq &2\|\psi_k(hD_x)(gu)\|_{L^2(\mathbb{T})}^2+C(2^kh)^2\|u(t)\|_{L^2(\mathbb{T})}^2,
\end{split}
\end{equation*}
thus
\begin{equation*}
\begin{split}
\sum_{k\leq \log_2(\epsilon_0/h)}\|\psi_k(hD_x)u(0)\|_{L^2(\mathbb{T})}^2\leq & \:  C\sum_{k\leq \log_2(\epsilon_0/h)}\int_0^{T_0}\|\psi_k(hD_x)(gu(t))\|_{L^2(\mathbb{T})}^2\\+ &\, CT_0\sum_{k\leq \log_2(\epsilon_0/h)}(2^kh)^2\|u(0)\|_{L^2(\mathbb{T})}^2\\
\leq & \,C\int_0^{T_0}\|gu(t)\|_{L^2(\mathbb{T})}^2dt+CT_0\epsilon_0^2\|u(0)\|_{L^2(\mathbb{T})}^2.
\end{split}
\end{equation*}
Therefore,
$$ \|u(0)\|_{L^2(\mathbb{T})}^2\leq C\int_0^{T_0}\int_{\mathbb{T}}|g(x)u(t,x)|^2dxdt+CT_0\epsilon_0^2\|u(0)\|_{L^2(\mathbb{T})}^2+C\|u(0)\|_{H^{-1}(\mathbb{T})}^2. 
$$
To complete the proof, we choose $\epsilon_0^2<\frac{CT_0}{2}$ and  \eqref{semiclassicalobservability} follows.\\

In summary, we have showed that  in order to prove the uniform observability inequality \eqref{observabilityuniform2} for all solutions of \eqref{1Dequation}, it suffices to prove the observability \eqref{singlefrequency} for all solutions of \eqref{semiclassical1}, uniformly in $0<h\ll 1$ and $k\in\mathbb{Z}$ such that $2^kh<\epsilon_0$.

\subsubsection{\bf Propagation estimate with parameter dependence symbol}
This section is devoted to the proof of Proposition \ref{singlefrequencyproposition}.
We recall some basic notations and results about $\tilde{h}-$pseudo-differential calculus. For $m\in\mathbb{R}$, let $S^m$ be the set of $\tilde{h}$-dependent functions $a(x,\xi,\tilde{h})$ with parameter $\tilde{h}\in(0,1)$ such that for any indices $\alpha,\beta$,
$$ \sup_{(x,\xi,\tilde{h})\in\mathbb{R}^{2d}\times(0,1)}|\partial_x^{\alpha}\partial_{\xi}^{\beta}a(x,\xi,\tilde{h})|\leq C_{\alpha,\beta}(1+|\xi|)^{m-|\beta|}.
$$  
For $a\in S^m$, we denote by $\mathrm{Op}_{\tilde{h}}(a)$ the $\tilde{h}-$pseudo-differential operator acting on Schwartz functions via
$$ \mathrm{Op}_{\tilde{h}}(a)f(x):=\frac{1}{(2\pi \tilde{h})^d}\int_{\mathbb{R}^{2d}}e^{\frac{i(x-y)\cdot\xi}{\tilde{h}}}a(x,\xi,\tilde{h})f(y)dyd\xi.
$$
 We refer \cite{booksemiclassical} for symbolic calculus and another basic properties about $\tilde{h}-$pseudo-differential operator. For functions on a compact Riemannian manifold, we can also define $\tilde{h}-$pseudo-differential operator by using local coordinate and partition of unity.   

Now let us consider the following $\epsilon-$dependence symbols:
$$p_{\epsilon}(x,\xi)=\left(\frac{\epsilon^4}{\xi}-\xi^3\right)\chi(\xi),\quad q_{\epsilon}(x,\xi)=\left(\frac{1}{\xi}-\epsilon^4\xi^3\right)\chi(\xi),
$$
where $\chi\in C_c^{\infty}(\mathbb{R})$ with supp$(\xi)\subset\{\alpha<|\xi|<\beta\}$ for some $0<\alpha<\frac{1}{2}$, $\beta>2$ and $\chi\equiv 1$ in a neighborhood of $\{1/2\leq|\xi|\leq 2\}$. Denote $P_{\epsilon}=\mathrm{Op}_{\widetilde{h}}(p_{\epsilon})$ and $Q_{\epsilon}=\mathrm{Op}_{\widetilde{h}}(q_{\epsilon})$. Denote $U_{\epsilon}(t)$ and $V_{\epsilon}(t)$  solutions of the operator equations
\begin{align} \label{solutionoperator1}
\begin{cases}    \frac{\widetilde{h}}{i}\partial_tU_{\epsilon}(t)+U_{\epsilon}(t)P_{\epsilon}=0,
\\
U_{\epsilon}(0)=I,
\end{cases}
\end{align}
\begin{align} \label{solutionoperator2}
\begin{cases}    \frac{\widetilde{h}}{i}\partial_tV_{\epsilon}(t)+V_{\epsilon}(t)Q_{\epsilon}=0,
\\
V_{\epsilon}(0)=I.
\end{cases}
\end{align}
The flows associated to the vector fields $H_{p_{\epsilon}},H_{q_{\epsilon}}$ are explicitly given by
$$ \phi_{\epsilon,t}(x_0,\xi_0)=\left(x_0-\left(
\frac{\epsilon^4}{\xi_0^2}+3\xi_0^2\right)\chi(\xi_0)t+\left(\frac{\epsilon^4}{\xi_0}-\xi_0^3\right)\chi'(\xi_0)t,\xi_0\right),
$$
$$\varphi_{\epsilon,t}(x_0,\xi_0)=\left(x_0-{\left(\frac{1}{\xi_0^2}+3\epsilon^4\xi_0^2\right)}\chi(\xi_0)t+\left(\frac{1}{\xi_0}-\epsilon^4\xi_0^3\right)\chi'(\xi_0)t,\xi_0\right)
$$
with respectively.

From Egorov's Theorem (see \cite{booksemiclassical}), for any symbol $a(x,\xi)\in C_c^{\infty}(T^*M)$, 
$$ U_{\epsilon}(-t)\mathrm{Op}_{\widetilde{h}}(a)U_{\epsilon}(t)=\mathrm{Op}_{\widetilde{h}}(a\circ\phi_{\epsilon,t})+O_{L^2\rightarrow L^2}(\widetilde{h}),
$$
$$V_{\epsilon}(-t)\mathrm{Op}_{\widetilde{h}}(a)V_{\epsilon}(t)=\mathrm{Op}_{\widetilde{h}}(a\circ\varphi_{\epsilon,t})+O_{L^2\rightarrow L^2}(\widetilde{h}).
$$
We remark that the bound $O_{L^2\rightarrow L^2}(\widetilde{h})$ is independent for $\epsilon\leq 1$ since all the semi-norms of the symbol $p_{\epsilon},q_{\epsilon}$ can be chosen continuously depending on $\epsilon$. \\

Now we prove the following localized observability estimates:
\begin{proposition}\label{highfrequency}
	There exist $C_0>0,T_0>0,\widetilde{h}_0>0$ such that for all $u_0\in L_0^2(\mathbb{T})$, all $\widetilde{h}\leq  \widetilde{h}_0$
	\begin{equation}\label{veryhigh}
	\|\psi(\widetilde{h}D_x)u_0\|_{L^2(\mathbb{T})}^2\leq C_0\int_0^{T_0}\|gU_{\epsilon}(t)\psi(\widetilde{h}D_x)u_0\|_{L^2(\mathbb{T})}^2dt,
	\end{equation}
	\begin{equation}\label{lesshigh}
	\|\psi(\widetilde{h}D_x)u_0\|_{L^2(\mathbb{T})}^2\leq C_0\int_0^{T_0}\|gV_{\epsilon}(t)\psi(\widetilde{h}D_x)u_0\|_{L^2(\mathbb{T})}^2dt.
	\end{equation}
\end{proposition}
\begin{proof}
	Here we only prove the first inequality, since the second one  follows in the same manner.
		Consider the symbol $a(x,\xi)=g(x)^2\widetilde{\psi}(\xi)$\footnote{Strictly speaking, $g$ is not smooth and we need approximate it by smoothing functions.} and its quantization $\mathrm{Op}_{\widetilde{h}}(a)=(g(x))^2\widetilde{\psi}(\widetilde{h}D_x)$, where $\widetilde{\psi}$ is a slight enlargement of $\psi$ such that $\widetilde{\psi}\psi=\psi$ and supp $\widetilde{\psi}\subset\{\alpha<|\xi|<\beta\}$. From Egorov's Theorem, we have
	$$ U_{\epsilon}(-t)\mathrm{Op}_{\widetilde{h}}(a)U_{\epsilon}(t)=\mathrm{Op}_{\widetilde{h}}(a\circ\phi_{\epsilon,t})+O_{L^2\rightarrow L^2}(\widetilde{h}),\textrm{ uniformly in } \epsilon\leq 1.
	$$
	Note that on the support of $a$, $\chi'(\xi)=0$, so we have $$\phi_{\epsilon,t}(x_0,\xi_0)=\left(x_0-\left(\frac{\epsilon^4}{\xi_0^2}+3\xi_0^2\right)t,\xi_0\right).
	$$
	
	Notice that $\left|\frac{\epsilon^4}{\xi_0^2}+3\xi_0^2\right|\geq c_0>0$, uniformly in $\epsilon,$ on the $\xi-$support of $\widetilde{\psi}$. Therefore, for some $T_0=T_0(c_0)>0$, and $c_1>0$ , we have
	$$ \int_0^{T_0}a\circ\phi_{\epsilon,t}dt\geq c_1>0.
	$$
	Now we calculate
	\begin{equation*}
	\begin{split}
	&\int_0^{T_0}\|gU_{\epsilon}(t)\psi(\widetilde{h}D_x)u_0\|_{L^2(\mathbb{T})}^2dt\\=&\int_0^{T_0}\left(gU_{\epsilon}(t)\psi(\widetilde{h}D_x)u_0,gU_{\epsilon}(t)\widetilde{\psi}(\widetilde{h}D_x)\psi(\widetilde{h}D_x)u_0\right)_{L^2(\mathbb{T})}dt\\
	=&\int_0^{T_0}\left(U_{\epsilon}(-t)\widetilde{\psi}(\widetilde{h}D_x)g^2U_{\epsilon}(t)u_0,\psi(\widetilde{h}D_x)u_0\right)_{L^2(\mathbb{T})}dt\\
	=&\left(\mathrm{Op}_{\widetilde{h}}(b_{T_0})\psi(\widetilde{h}D_x)u_0,\psi(\widetilde{h}D_x)u_0\right)_{L^2(\mathbb{T})},
	\end{split}
	\end{equation*}
	with $b_{T_0}(x,\xi)=\int_0^{T_0}a\circ\phi_{\epsilon,t}dt$ modulo $\widetilde{h}S^0$. Thus, from Sharp G\aa rding inequality (see  \cite{booksemiclassical}), we have
	$$\left(\mathrm{Op}_{\widetilde{h}}(b_{T_0})\psi(\widetilde{h}D_x)u_0,\psi(\widetilde{h}D_x)u_0\right)_{L^2(\mathbb{T})}\geq \frac{c_1}{2}\|\psi(\widetilde{h}D_x)u_0\|_{L^2(\mathbb{T})}^2-C\widetilde{h}\|\psi(\widetilde{h}D_x)u_0\|_{L^2(\mathbb{T})}^2.
	$$
	To conclude the proof, we choose $\widetilde{h}_0<\min\{\frac{c_1}{4C},1\}$.
\end{proof}

\begin{proof}[Proof of Proposition \ref{singlefrequencyproposition}]
For fixed $h\ll 1$, we analyse the three regimes for $k\in\mathbb{Z}$:\\

\emph{\textbf{Case 1}: $|k|\leq N_0$ for some large natural number $N_0$}\\
This  corresponds to the case $|\xi|\sim 1$. Let $u_k=\psi_k(hD_x)u$, the equation satisfied by $u_k$ is  \eqref{semiclassical1}. We can either use \eqref{veryhigh} or \eqref{lesshigh} with parameter $\epsilon=1$ to obtain that (note that $\widetilde{h}=2^k\widetilde{h}\sim h$ in this regime)
$$\|\psi_k(hD_x)u_0\|_{L^2(\mathbb{T})}^2\leq C_0\int_0^{T_0}\|g\psi_k(hD_x)u(t)\|_{L^2(\mathbb{T})}^2dt.
$$

\emph{\textbf{Case 2}: $k\leq -N_0$ for some large constant $N_0$}\\
This case corresponds to $|\xi|\sim 2^{-k}\gg 1$. Defining a new semi-classical parameter $\widetilde{h}_k=2^kh\ll 1$ and to rescale the time variable we set $w_k(t,x):=\psi(\widetilde{h}_kD_x)u(2^{2k}t,x)$ and $u_k=\psi(\widetilde{h}_kD_x)u$. The equation satisfied by $w_k$ is:
$$ \widetilde{h_k}\partial_tw_k+(\widetilde{h_k}\partial_x)^3w_k+2^{4k}(\widetilde{h}\partial_x)^{-1}w_k=0.
$$
Applying  \eqref{veryhigh} to $w_k$ with $\epsilon=2^k\ll 1$ and $\widetilde{h}=\widetilde{h}_k$ we obtain
$$ \|w_k(0)\|_{L^2(\mathbb{T})}^2\leq C\int_0^{T_0}\|gw_k(t)\|_{L^2(\mathbb{T})}^2dt.
$$
From conservation of $L^2$ norm, we apply the inequality above $2^{-2k}-1$ times and obtain that
\begin{equation*}
\begin{split}
\frac{1}{2^{2k}}\|u_k(0)\|_{L^2(\mathbb{T})}^2\leq &\, \frac{C}{2^{2k}}\sum_{M=0}^{2^{-2k}-1}\int_{M2^{2k}T_0}^{(M+1)2^{2k}T_0}\|gu_k(t)\|_{L^2(\mathbb{T})}^2dt\\
=\, &\frac{C}{2^{2k}}\int_0^{T_0}\|gu_k(t)\|_{L^2(\mathbb{T})}^2dt.
\end{split}
\end{equation*}
This is exactly
$$ \|\psi_k(hD_x)u(0)\|_{L^2(\mathbb{T})}^2\leq C\int_0^{T_0}\|g\psi_k(hD_x)u(t)\|_{L^2(\mathbb{T})}^2dt.
$$

\emph{\textbf{Case 3:} $k\geq N_0$}\\
This case corresponds to $|\xi|\sim 2^{-k}\ll 1$. Define the new small semi-classical parameter $\tilde{h}_k=2^{k}h$. The $\widetilde{h}-$ pseudo differential calculus applies,  by the restriction $2^{k}h\leq \epsilon_0\ll 1$.

Denote by $u_k=\psi(\widetilde{h}_kD_x)u$ and define $v_k(t,x)=u_k(2^{-2k}t,x)$. $v_k$ solves the equation
$$ \widetilde{h}_k\partial_tv_k+2^{-4k}(\widetilde{h}_k\partial_x)^3v_k+(\widetilde{h}_k\partial)^{-1}v_k=0.
$$
Applying \eqref{lesshigh} with $\widetilde{h}=\widetilde{h}_k,\epsilon=2^{-k}$, we obtain that
$$ \|v_k(0)\|_{L^2(\mathbb{T})}^2\leq C\int_0^{T_0}\|gv_k(t)\|_{L^2(\mathbb{T})}^2dt.
$$

Again by conservation of $L^2$ norm as in the argument of Case 2, we finally have
$$ \|u_k(0)\|_{L^2(\mathbb{T})}^2\leq C\int_0^{T_0}\|gu_k(t)\|_{L^2(\mathbb{T})}^2dt.
$$
This completes the proof of Proposition \ref{singlefrequencyproposition}. Hence the proof of Proposition \ref{1Dobservability}, and the observability inequality \eqref{observability2D} for the lineariaed KP-II equation are also complete.
\end{proof}

As a consequence of Proposition \ref{HUM} , the internal controllability for the linear KP II is obtained. We conclude this section by summarizing it in  the following proposition:

\begin{proposition}\label{controloperator} Given $T>0$,  there exists a bounded linear operator 
	$$ \Upsilon:  (L_0^{2}(\T^2))^{2} \to  L^{2}(0,T;L^{2}(\T^2)) $$
	such that for any $u_{0}, u_{1} \in L_0^{2}(\T^2)$, the control defined by $h:= \Upsilon(u_{0},u_{1})$ drives the solution of 
	\begin{align} \label{controlKP}
	\begin{cases}    \partial_tu+\partial_x^3u+\partial_x^{-1}\partial^2_{y}u=\mathcal{G}h,\quad
	(t,x)\in\R\times\mathbb{T}^{2},
	\\
	u|_{t=0}=u_0,
	\end{cases}
	\end{align}
	
	to $u(T)=u_{1}$.\\
	Moreover, we have 
	$$\|\Upsilon(u_{0},u_{1}) \|_{ L^{2}(0,T;L^{2}(\T^2)) } \le C \|(u_{0},u_{1})\|_{(L^{2}(\T^2))^{2}}.$$
\end{proposition}

\section{Local controllability of Nonlinear equation}
For the  full KP-II control system 
\begin{align} \label{controlnonlinear}
\begin{cases}    \partial_tu+\partial_x^3u+\partial_x^{-1}\partial^2_{y}u + u\partial_{x}u=\mathcal{G}h,\quad
(t,x)\in\R\times\mathbb{T}^{2},
\\
u|_{t=0}=u_0, \, u|_{t=T}=u_1,
\end{cases}
\end{align} 
in order to prove the existence of $u\in L^{2}(0,T;L_0^{2}(\T^2))$ solving $u|_{t=0}=u_0, \, u|_{t=T}=u_1$,  we will reduce it to a fixed point problem by standard argument.
\begin{proof}[Proof of Theorem \ref{nonlinear} ]
 The solution of \eqref{controlnonlinear} with control input $h$ is given by
$$ u(t)=S(t)u_0+\upsilon(t,u)+\int_0^tS(t-t')\mathcal{G}h(t')dt'
$$
with
$$\upsilon(t,u)= \int_{0}^{t} S(t-t') u \partial_{x}u dt'.$$
It must satisfy
$$ u_1=S(T)u_0+v(T,u)+\int_0^TS(T-t')\mathcal{G}h(t')dt.
$$
Choosing the control input of the form $h=\Upsilon(u_0,w)$, this implies that
$$ S(T)u_0+\int_0^TS(T-t')\mathcal{G}h(t')dt'=w.
$$
This indicates that $w=u_1-\upsilon(T,u)$. In summary, defining the nonlinear map $\Gamma$ by
$$\Gamma(u)=S(t)u_0+\upsilon(t,u)+\int_{0}^{t}S(t-t')\mathcal{G}h_u(t')dt'
$$
with
$$ h_u=\Upsilon(u_0,u_1-v(T,u)),
$$
we need to find a fixed point of $\Gamma$.

We need show that $\Gamma:  X_T^{0,\frac{1}{2},b_1}\cap Z_T^{0,\frac{1}{2}}\rightarrow X_T^{0,\frac{1}{2},b_1}\cap Z_T^{0,\frac{1}{2}}$ is a contraction in a bounded ball. From Proposition \ref{bilinear} and Proposition \ref{linearestimate}, we have  
\begin{align*}
\|\Gamma(u)\|_{X_T^{0,\frac{1}{2},b_1}\cap Z_T^{0,\frac{1}{2}}} &\le C \left(  \|u_{0}\|_{L^2(\T^2)} + \|  \mathcal{G}h_{u} \|_{X_T^{0,-\frac{1}{2},b_{1}}}+ \|u\|^{2}_{X_T^{0,\frac{1}{2},b_1}}\right)\\
& \le C \left(  \|u_{0}\|_{L^{2}(\T^2)} + \|u_{1}\|_{L^{2}(\T^2)}  + \| \upsilon(T,u)(T) \|_{L^2(\T^2)}+ \|u\|^{2}_{X_T^{0,\frac{1}{2},b_1}} \right)\\
& \le C \left(  \|u_{0}\|_{L^{2}(\T^2)} + \|u_{1}\|_{L^{2}(\T^2)} + \|u\|^{2}_{X_T^{0,\frac{1}{2},b_1}}\right),
\end{align*}
where $C>0$ does  not depend on $u_{0}$.
For $R>0$, let  $B_{R}=B_{R}(0)$ be the ball centered at zero with radio $R$, that is 
$$B_{R}:= \{ u\in X_T^{0,\frac{1}{2},b_1}\cap Z_T^{0,\frac{1}{2}} :  \| u\|_{X_T^{0,\frac{1}{2},b_1}\cap Z_T^{0,\frac{1}{2}}}<R\}.$$
Then
\begin{gather}\label{contraccion_cond1}
\|\Gamma(u)\|_{X_T^{0,\frac{1}{2},b_1}\cap Z_T^{0,\frac{1}{2}}}\le C \left(  \|u_{0}\|_{L^{2}(\T^2)} +\|u_{1}\|_{L^{2}(\T^2)} + R^{2}\right).
\end{gather}
Additionally, for $u, v\in B_{R}$ we have 
\begin{align}\label{contraccion_cond2}
\|\Gamma(u)-\Gamma(v)\|_{X_T^{0,\frac{1}{2},b_1}\cap Z_T^{0,\frac{1}{2}}}\le & C   \left\| \int_{0}^{t} S(t-\tau)  (\mathcal{G}h_{u}- 
\mathcal{G}h_{v}) d t'\right\|_{X_T^{0,\frac{1}{2},b_1}\cap Z_T^{0,\frac{1}{2}}}\notag\\+ 
&\left\| \int_{0}^{t} S(t-t') (u\partial_{x}u- v\partial_{x}v)d t'\right\|_{X_T^{0,\frac{1}{2},b_1}\cap Z_T^{0,\frac{1}{2}}}  \notag\\ 
\le & C   \left\| \Upsilon(u_{0}, u_{1}-\upsilon(T,u)) - 
\Upsilon(u_{0}, u_{1}-\upsilon(T,v))) \right\|_{X_T^{0,\frac{1}{2},b_1}\cap Z_T^{0,\frac{1}{2}}} \notag\\+& C\left\| \int_{0}^{t} S(t-t') (u\partial_{x}u- v\partial_{x}v)d t'\right\|_{X_T^{0,\frac{1}{2},b_1}\cap Z_T^{0,\frac{1}{2}}} \notag \\
\le & C   \| \upsilon(T,u) - \upsilon(T,v) \|_{X_T^{0,\frac{1}{2},b_1}\cap Z_T^{0,\frac{1}{2}}} \notag\\+ 
&\|  u- v \|_{X_T^{0,\frac{1}{2},b_1}} \|  u+ v \|_{X_T^{0,\frac{1}{2},b_1}} \notag\\
\le & C  \|  u- v \|_{X_T^{0,\frac{1}{2},b_1}} \|  u+ v \|_{X_T^{0,\frac{1}{2},b_1}} \notag\\ \le &\frac12  \|  u- v \|_{X_T^{0,\frac{1}{2},b_1}}
\end{align}
by using properties of the bounded linear operator $\Upsilon$.  Choosing  $\delta>0$ and $R>0$ such that  $ 2C\delta + CR^{2} \le R$ and $CR<\frac12$ with 
$\|u_{0}\|_{L^{2}(\T^2)}<\delta$ and  $\|u_{1}\|_{L^{2}(\T^2)}<\delta$.  We can conclude from \eqref{contraccion_cond1}  that the image of $B_{R}$ through $\Gamma$ stays in the ball  $B_{R}$ and from \eqref{contraccion_cond2} that $\Gamma$ is a contraction. The proof of Theorem \ref{nonlinear} is complete.
\end{proof}
%Strategy: Fix-point argument.
%Consider the problem
%\begin{align} \label{nonlinear1}
%\begin{cases}    \partial_tu+iLu+\mathcal{N}(u)=\mathcal{G}h_1+\mathcal{G}h_2,
%\\
%u|_{t=0}=u_0\in L^2, u|_{t=T}=u_T\in L^2,
%\end{cases}
%\end{align}
%We first select $h_1\in L^2((0,T);L^2)$ so that
%\begin{align} 
%\begin{cases}    \partial_tv+iLv=\mathcal{G}h_1,
%\\
%u|_{t=0}=u_0\in L^2, u|_{t=T}=u_T\in L^2.
%\end{cases}
%\end{align}
%One need show that there exists $\delta>0$ small so that if $\|u_0\|_{L^2}\leq\delta,\|u_{T}\|_{L^2}\leq\delta$, then for any $h_2\in L^2((0,T);L^2)$ with $\|h_2\|_{L^2((0,T);L^2)}\ll 1$ small, there exists a unique solution $u\in X_T$ (certain space-time Banach space) and $\|u\|_{X_T}\leq C(\|h_2\|_{L^2((0,T);L^2)})$ is also small.
%
%Next we write $u=v+w$, and $w$ solves
%\begin{align} 
%\begin{cases}    \partial_tw+iLw+\mathcal{N}(u)=\mathcal{G}h_2,
%\\
%u|_{t=0}=0, u|_{t=T}=0.
%\end{cases}
%\end{align}
%Now let $z$ be the solution to
%\begin{align} 
%\begin{cases}    \partial_tz+iLz-\mathcal{N}(u)=0,
%\\
% z|_{t=T}=0.
%\end{cases}
%\end{align}
%and we define the nonlinear operator $A:h_2\mapsto z(0)$.
%
%Next we write $k=w+z$, and then $k$ solves
%\begin{align} 
%\begin{cases}    \partial_tk+iLk=\mathcal{G}h_2,
%\\
%k|_{t=0}=z(0), u|_{t=T}=0.
%\end{cases}
%\end{align}
%Let $\Lambda$ be the HUM operator of the linear control problem, the nonlinear control problem reduces to find  the fix-point of the operator $\Lambda A$. 

\section{Non Controllability in horizontal strip}
In this section, we prove Theorem \ref{negative} by disproving the observability for the linearized KP-II equation \eqref{exactcontrollinear} on the horizontal control region. By translation, we may assume that the horizontal control region is $\omega=(-\pi,-\alpha)\cup (\alpha,\pi]$ for some $0<\alpha<\pi$. Recall that $\mathcal{K}$ is defined by \eqref{horizontalcontrol}. By HUM method, the proof of Theorem \ref{negative} reduces to prove the following:
\begin{proposition}\label{noncontrollability}
For any $T>0$, the observability inequality
	\begin{equation}\label{kobservability} \|u(0)\|_{L^2(\mathbb{T}^2)}^2\leq C_T\int_0^T\int_{\mathbb{T}^2}|\mathcal{K}u(t,x,y)|^2dxdydt 
	\end{equation}
	does not hold for every solution $u\in L^2((0,T);L_0^2(\T^2)$ of the linearized KP-II equation 
	$$\partial_tu+\partial_x^3u+\partial_x^{-1}\partial_y^2u=0.
	$$
\end{proposition}
The building block for proving Proposition \ref{noncontrollability} is the following lemma for 1D semi-classical Schr\"odinger equation: 
\begin{lemma}\label{counterexampleLS}
	Assume that $\omega=(-\pi,-\alpha)\cup(\alpha,\pi]$ for  $0<\alpha<\pi$. Then for any $T>0$, there exists a sequence of solutions $u_n$ to
	\begin{align} \label{1DLS}
	\begin{cases} 
	ih_n\partial_tu_n+h_n^2\partial_x^2u_n=0,
	\\
	u_n|_{t=0}=u_{n,0}\in L^2(\mathbb{T}),
	\end{cases}
	\end{align}
	such that 
	$$\liminf_{n\rightarrow\infty}\|u_{n,0}\|_{L^2(\mathbb{T})}>0$$
	and
	$$ \lim_{n\rightarrow\infty}\int_0^T\int_{\omega}|u_n(t,x)|^2dxdt=0.
	$$
\end{lemma}
\begin{proof}
	Take $G(x)=e^{-\frac{x^2}{2}}$ and define $G^{\epsilon_n}(x)=\frac{1}{\sqrt{\epsilon_n}}G\left(\frac{x}{\epsilon_n}\right)$. Denote the Fourier coefficient of $G^{\epsilon_n}$ by $$g^{\epsilon_n}(k)=\frac{1}{2\pi}\int_{-\pi}^{\pi}G^{\epsilon_n}(x)e^{-ikx}dx=\frac{\sqrt{\epsilon_n}}{2\pi}\int_{-\frac{\pi}{\epsilon_n}}^{\frac{\pi}{\epsilon_n}}G(z)e^{-i\epsilon_nkz}dz.$$
	The coefficient function $g^{\epsilon_n}(z)$ satisfies the following estimates:
	\begin{equation}\label{coefficient}
	\|g^{\epsilon_n}\|_{L^{\infty}(\mathbb{R})}=O(\epsilon_n^{1/2}), \, \|(g^{\epsilon_n})'\|_{L^{\infty}(\mathbb{R})}=O(\epsilon_n^{3/2}), \, \|(g^{\epsilon_n})''\|_{L^{\infty}(\mathbb{R})}=O(\epsilon_n^{5/2}).
	\end{equation}

	Take an even cut-off function $\psi\in C_c^{\infty}(\mathbb{R})$ with supp $\psi\subset [-B,B]$  with $0<b<B$ and $0\leq \psi\leq 1$, $\psi(z)\equiv 1$, for all $|z|\leq b$.  We define
	$$ u_{n,0}(x)=\sum_{k\in\mathbb{Z}}g^{\epsilon_n}(k)\psi(h_nk)e^{ikx},
	$$
	and then the corresponding solution to \eqref{1DLS} is given explicitly by
	$$ u_n(t,x)=\sum_{k\in\mathbb{Z}}g^{\epsilon_n}(k)\psi(h_nk)e^{i(kx-k^2h_nt)}.
	$$
	We need estimate the mass of initial data. Firstly, 
	$$ \|G^{\epsilon_n}\|_{L^2(\mathbb{T})}^2=\sum_{k\in\mathbb{Z}}|g^{\epsilon_n}(k)|^2\sim 1
	$$
	holds from Plancherel Theorem and the definition of $g^{\epsilon_n}(k)$. We next estimate the mass away from the frequency scale $h_n^{-1}$, that is 
	\begin{equation*}
	\begin{split}
	\sum_{k\in\mathbb{Z}}\left|(1-\psi(h_nk))g^{\epsilon_n}(k)\right|^2\leq &\sum_{|k|>h_n^{-1}b}|g^{\epsilon_n}(k)|^2\\
	\leq &\sum_{|k|>h_n^{-1}b}\frac{\epsilon_n}{4\pi^2}\left|\int_{\mathbb{R}}G(z)e^{-ik\epsilon_n}zdz\right|^2\\
	=&\sum_{|k|>h_n^{-1}b}\frac{\epsilon_n}{4\pi^2}\left|\int_{\mathbb{R}}G(z)\frac{1}{-ik\epsilon_n}\frac{d}{dz}e^{-ik\epsilon_n}zdz\right|^2\\
	\leq& \sum_{|k|>h_n^{-1}b}\frac{1}{4k^2\pi^2\epsilon_n}\|G'\|_{L^1(\mathbb{R})}^2.
	\end{split}
	\end{equation*}
	By setting $\epsilon_n=\sqrt{h_n}\ll 1$, we have $\|(1-\psi(h_nD_x))G^{\epsilon_n}\|_{L^2(\mathbb{T})}\ll 1$ and then $\|u_{n,0}\|_{L^2(\mathbb{T})}\sim 1$.
		It remains to estimate the term on the right hand side of observability inequality \eqref{kobservability}.
	
	Observe that $u_{n,0}$ is localized by $|k|\leq \frac{B}{h_n}$ in frequency and by $|x|\leq \epsilon_n$ in space obeying uncertain principle ($\epsilon_n h_n^{-1}\gtrsim 1$). Since the wave packet of the frequency scale smaller than $Bh_n^{-1}$ moves at velocity bigger than $ 2Bh_n^{-1}$, it will remain small for $|t|<T$ in $\omega$. More precisely, we need a decay estimate for $|u_n(t,x)|$ when $x\in\omega$ and $|t|<T$. Now we choose $B>0$ such that $|x-2Bt|\geq c_0>0$ mod $2\pi$ for all $x\in\omega$ and $|t|\leq T$.
	Write
	$$ u_n(t,x)=\sum_{k\in\mathbb{Z}}K_{t,x}^{(n)}(k)
	$$
	with
	$$ K_{t,x}^{(n)}(z)=g^{\epsilon_n}(z)\psi(h_nz)e^{i(zx-h_nz^2t)}.
	$$
	From Poisson summation formula, we have
	$$ u_n(t,x)=\sum_{m\in\mathbb{Z}}\widehat{K_{t,x}^{(n)}}(2\pi m).
	$$
	For fixed $m\in\mathbb{Z}$, 
	\begin{equation*}
	\begin{split}
	\widehat{K_{t,x}^{(n)}}(2\pi m)=&\int_{\mathbb{R}}g^{\epsilon_n}(z)\psi(h_nz)e^{i\varphi_{t,x}(z)}dz\\
	=&\int_{\mathbb{R}}g^{\epsilon_n}(z)\psi(h_nz)\mathcal{L}^2(e^{i\varphi_{t,x}(z)})dz
	\end{split}
	\end{equation*}
	with $\displaystyle{\mathcal{L}=\frac{1}{i\varphi'_{t,x}(z)}\frac{d}{dz}}$ and $\varphi_{t,x}(z)=(x-2\pi m)z-h_nz^2t$. By integration by parts, we have
	\begin{equation*}
	\begin{split}
	\widehat{K_{t,x}^{(n)}}(2\pi m)=\int_{\mathbb{R}}\frac{d}{dz}\left(\frac{1}{i\varphi'_{t,x}(z)}\frac{d}{dz}\left(\frac{g^{\epsilon_n}(z)\psi(h_nz)}{i\varphi'_{t,x}(z)}\right)\right)e^{i\varphi_{t,x}(z)}dz.
	\end{split}
	\end{equation*}
	After tedious calculation, we obtain that
	\begin{equation*}
	\begin{split}
	&\frac{d}{dz}\left(\frac{1}{i\varphi'_{t,x}(z)}\frac{d}{dz}\left(\frac{g^{\epsilon_n}(z)\psi(h_nz)}{i\varphi'_{t,x}(z)}\right)\right)\\ =&\frac{(g^{\epsilon_n})''\psi(h_nz)+2h_n(g^{\epsilon_n})'\psi'(h_nz)+h_n^2\psi''(h_nz)g^{\epsilon_n}}{(\varphi'_{t,x})^2}\\
	-&\frac{3((g^{\epsilon_n})'\psi(h_nz)+h_n\psi'(h_nz)g^{\epsilon_n})\varphi''_{t,x}}{(\varphi'_{t,x})^3}
	-\frac{3g^{\epsilon_n}\psi(h_nz)(\varphi''_{t,x})^2}{(\varphi'_{t,x})^4}.
	\end{split}
	\end{equation*}
	From \eqref{coefficient}, we have 
	$$ |\widehat{K_{t,x}^{(n)}}(2\pi m)|\leq \sup_{|h_nz|\leq B} \frac{C\epsilon_n^{1/2}\|\psi\|_{W^{2,1}(\mathbb{R})}}{|(x-2h_nzt)-2\pi m|^2}.
	$$
	For any $x\in 2\pi p+(-\pi,-\alpha)\cup (\alpha,\pi]$, $|x-2h_nzt|\geq c_0>0$ mod $2\pi$ with $p\in \Z$, it holds
	\begin{equation*}
	\begin{split}
	\sum_{m\in\mathbb{Z}} |\widehat{K_{t,x}^{(n)}}(2\pi m)|\leq & C\sum_{m\in\mathbb{Z}}\frac{C\epsilon_n^{1/2}}{|c_0-2\pi(m-p)|^2}\\
	\leq & C\epsilon_n^{1/2}.
	\end{split}
	\end{equation*}
	Therefore,
	$$ \int_0^T\int_{\omega}|u_n(t,x)|^2dxdt\leq C\epsilon_n^{1/2}T|\omega|\rightarrow 0, \textrm{ as }n\rightarrow\infty.
	$$
	This completes the proof of Lemma \ref{counterexampleLS}.
\end{proof}
Now we are ready to prove Propsition \ref{noncontrollability}.
\begin{proof}[Proof of Proposition \ref{noncontrollability}]
	For any $T>0$, we will construct a sequence of solutions $u_n$ to the linearized KP-II equation such that 
	$$ \|u_n(0)\|_{L^2(\mathbb{T}^2)}\sim 1 \quad \text{and} \quad 
 \lim_{n\rightarrow\infty}\int_{0}^{T}\int_{\mathbb{T}^2}|\mathcal{K}u_n(t,x,y)|^2dxdydt=0.
	$$
	Denote by $v_n(t,y)$ the sequence of solutions to the semi-classical Schr\"odinger equation which satisfies the conditions in Lemma \ref{counterexampleLS}. Define
	$$ u_n(t,x,y)=v_n(t,y)e^{\frac{it}{h_n^3}}e^{\frac{ix}{h_n}}=\sum_{k\in\mathbb{Z}}\widehat{v_n}(k)e^{i(ky-h_nk^2t)}e^{i\left(\frac{x}{h_n}+\frac{t}{h_n^3}\right)}.
	$$
	Then $u_n$ solves the linearized KP-II equation. Moreover, 
	$$ \|u_n(0)\|_{L^2(\mathbb{T}^2)}=\|v_n(0)\|_{L^2(\mathbb{T})}\sim 1,
	$$
	and
	$$ \int_0^T\int_{\omega}|u_n(t,x,y)|^2dxdydt=\int_0^T\int_{(-\pi,\alpha)\cup(\alpha,\pi]}|v_n(t,y)|^2dtdy\rightarrow 0, \textrm{as }n\rightarrow\infty.
	$$
	
	Now we claim that
	$$ \lim_{n\rightarrow\infty}\int_{\mathbb{T}}g(y')v_n(t,y')dy'\rightarrow 0 \textrm{ in }L^{\infty}([0,T];L^2(\mathbb{T})).
	$$
	Indeed, 
	\begin{equation*}
	\begin{split}
	\left|\int_{\mathbb{T}}g(y')v_n(t,y')dy'\right|=&\left|\sum_{k\in\mathbb{Z}}\ov{\widehat{g}(k)}g^{\epsilon_n}(k)\psi(h_nk)e^{-ik^2t}\right|\\
	=&\left|\left(\sum_{|k|\leq M}+\sum_{|k|>M}\right)\ov{\widehat{g}(k)}g^{\epsilon_n}(k)\psi(h_nk)e^{-ik^2t}\right|\\
	\leq &\epsilon_n^{1/2}\|g\|_{L^2(\mathbb{T})}M^{1/2}+\|G^{\epsilon_n}\|_{L^2(\mathbb{T})}\left(\sum_{|k|>M}|\widehat{g}(k)|^2\right)^{1/2}
	\end{split}
	\end{equation*}
	and the right hand side tends to $0$ as $n\rightarrow\infty$ since we can choose $M$ to be arbitrarily large before taking the limit in $n$. The validity of the claim implies that $\displaystyle{g(y)\int_{\mathbb{T}}g(y')u_n(t,x,y')dy'\rightarrow 0}$ in $L^{2}([0,T]\times\mathbb{T}^2)$. This completes the proof of Proposition \ref{noncontrollability}, as well as Theorem \ref{negative}.
\end{proof}

\begin{center} 
	
\end{center}

\end{document}